\documentclass[leqno,11pt]{article}

\usepackage{amsmath, amssymb}
\usepackage{amsthm}

\theoremstyle{plain} 
\newtheorem{theorem}{\noindent\sc Theorem}[section]
\newtheorem{lemma}[theorem]{\noindent\sc Lemma}

\newtheorem{proposition}[theorem]{\noindent\sc Proposition}

\theoremstyle{definition}

\makeatletter

\renewcommand\section{\@startsection{section}{1}{\z@}%
                                  {-3.5ex \@plus -1ex \@minus -.2ex}%
                                  {2.3ex \@plus.2ex}%
                                  {\normalfont\S\large\bfseries}}
\makeatletter

\title{\large \uppercase {Prolongations of Lie algebra Representations}} 

\author{\small \bf{* H\"{u}lya Kad{\i}o\u{g}lu} \\ \small Department of Mathematics, Gazi University, Turkey\\ \small and Department of Mathematics, Idaho State University, USA \\ \small kayahuly@isu.edu and hulyakaya@gazi.edu.tr.\\ \small \bf{Erdo\u{g}an Esin}  \\ \small Department of Mathematics, Gazi University, Turkey\\ \small eresin@gazi.edu.tr. \\ \small \bf{Yusuf Yayl{\i}}\\ \small Department of Mathematics, Ankara University, Turkey.\\ \small Yusuf.Yayli@science.ankara.edu.tr}

\date{}
\begin{document}
   \maketitle
   \footnote{*Corresponding Author

}
   \begin{abstract}
 In this paper, we present a study on the prolongations of representations of Lie algebras. We show that a tangent bundle of a given Lie algebra attains a Lie algebra structure. Then, we prove that this tangent bundle is algebraically isomorphic to the Lie algebra of a tangent bundle of a Lie group. Using these, we define prolongations of representations of Lie algebras. We show that if a Lie algebra representation corresponds to a Lie group representation, then prolongation of Lie algebra representation corresponds to the prolonged Lie group representation. 

\end{abstract}
 \section{Introduction}

\hspace{5mm} In this paper, we present a study on the prolongations of representations of Lie algebras. In our previous work \cite{Myarticle}, we have obtained a basis for the tangent bundle of an arbitrary finite-dimensional vector space and shown that if a function is linear, then its tangent function is also linear. We have also  defined prolongations of finite-dimensional real representations of Lie groups and obtained faithful representations on tangent bundles of  Lie groups. The remaining part of the paper was dedicated to the study of the properties of these faithful representations.

   For any Lie group $G$, there is a natural action called adjoint action of $G$ on its Lie algebra $Lie(G)$. Therefore it is possible to form a semi direct product of $G$ with $Lie(G)$ and this semi direct product group is isomorphic to $TG$. Given a representation $\Phi$  of a Lie group $G$ acting on a vector space $V$, it is well known that there exists a naturally associated Lie algebra representation $\phi$ of $Lie(G)$, obtained by taking the differential of $\Phi$ at the identity \cite{Arv,Varadajan}. Therefore it is always possible to obtain a Lie algebra representation of $Lie(TG)$ by using prolongation of Lie group representations defined in our previous paper. In this study, we introduce a different and direct way to find a Lie algebra representation on $Lie(TG)$. First we show that the tangent bundle of a Lie algebra is also Lie algebra. Then, we prove that this tangent bundle is algebraically isomorphic to $Lie(TG)$ where $G$ represents a Lie group, $TG$ represents $G$'s tangent bundle, and $Lie(TG)$ represents the Lie algebra of $TG$. Using these, we define prolongations of representations of Lie algebras.  We show that if a Lie algebra representation corresponds to a Lie group representation, then prolongation of Lie algebra representation corresponds to the prolonged Lie group representation.

 \section{Lie algebra Structure on the Tangent Bundle of a Lie Algebra}

\hspace{5mm} In this section, we present our original findings. Necessary details are given in the appendix. 

\begin{proposition}\label{newprop2}  $(T(Lie(G)),\oplus,\bullet,T\varphi )$ is a Lie algebra, where $\oplus$ and $\bullet$ is sum and scalar product-defined in \cite{Morimoto:1968}- of the vector space $T(Lie(G))$ respectively.
\end{proposition}

\begin{proof} Showing that $T(\varphi)$, where $\varphi$ is the Lie bracket operator, satisfies bilinearity, antisymmetry and Jacobi identity, we complete the proof. Here, $f_a:Lie(G)\to Lie(G)\times Lie(G)$ defined as $f_a(b)=(a,b)$ and $\bar{f}_a:Lie(G)\to Lie(G)$ defined as $\bar{f}_a(b)=(b,a)$ where $a,b\in Lie(G)$. Using (\ref{eq1}) and (\ref{eq2}) for all $\ (X_a,Y_b)\in T(Lie(G))\times T(Lie(G))$, we have
\begin{eqnarray}
T\varphi(X_a,Y_b)&=&T\varphi(T\bar{f}_b(X_a)+Tf_a(Y_b))\nonumber \\
&=&T(\varphi \circ \bar{f}_b)(X_a)+T(\varphi \circ f_a)(Y_b)\nonumber \\
&=&T(\sigma_{-1} \circ \varphi \circ f_b)(X_a)+T(\sigma_{-1} \circ \varphi \circ \bar{f}_a)(Y_b)\nonumber \\
&=&T\sigma_{-1}(T\varphi(Y_b,X_a))\nonumber \\
&=&-1\bullet T\varphi(Y_b,X_a).\label{hkk3}
\end{eqnarray}
(\ref{hkk3}) shows that $T\varphi$ is antisymmetric. For all $X_a,Y_b,Z_c \in T(Lie(G))$ and $\lambda \in \mathbb{R}$, Eqs. (\ref{eq3}) and (\ref{eq4}) give
\begin{eqnarray}
T\varphi(\lambda \bullet X_a,Y_b)&=&T\varphi(T\bar{f}_b(\lambda \bullet X_a)+Tf_{\lambda a})(Y_b)\nonumber \\
&=&T(\varphi \circ \bar{f}_b \circ \sigma_{\lambda})(X_a)+T(\varphi \circ f_{\lambda a})(Y_b)\nonumber\\
&=&T(\sigma_{\lambda} \circ \varphi \circ \bar{f}_b)(X_a)+T(\sigma_{\lambda} \circ \varphi \circ f_a)(Y_b)\nonumber\\
&=& \lambda \bullet T\varphi(X_a,Y_b).\label{hkk4}
\end{eqnarray}
Using Eqs. (\ref{eq5}) and (\ref{eq6}), we have
\begin{eqnarray}
T\varphi(X_a,Z_c)\oplus T\varphi(Y_b,Z_c)&=&T\tau_{\varphi(b,c)}(T\varphi(X_a,Z_c))+T\tau_{\varphi(a,c)}(T\varphi(Y_b,Z_c))\nonumber \\
&=& T(\tau_{\varphi(b,c)} \circ \varphi \circ \bar{f}_c)(X_a)+T(\tau_{\varphi(a,c)} \circ \varphi \circ \bar{f}_c)(Y_b)\nonumber \\
&+&T(\tau_{\varphi(b,c)}\circ \varphi \circ f_a +\tau_{\varphi(a,c)}\circ \varphi \circ f_b)(Z_c)\nonumber \\
&=&T(\varphi \circ \bar{f}_c \circ \tau_b)(X_a)+T(\varphi \circ \bar{f}_c \circ \tau_a)(Y_b)\nonumber \\
&+& T(\tau_{\varphi(b,c)}\circ \varphi \circ f_a)(Z_c)+T(\tau_{\varphi(a,c)}\circ \varphi \circ f_b)(Z_c)\nonumber \\
&=&T\varphi(T\bar{f}_c(X_a\oplus Y_b)+Tf_{a+b}(Z_c))\nonumber \\
&=& T\varphi(X_a \oplus Y_b,Z_c).\label{hkk5}
\end{eqnarray}
(\ref{hkk4}) and (\ref{hkk5}) show that $T\varphi $ is a bilinear function. Finally, using (\ref{eq7}), (\ref{eq8}), and (\ref{eq9}) we have 
\begin{eqnarray}
&&T\varphi(T\varphi(X_a,Y_b),Z_c) \oplus T\varphi(T\varphi(Y_b,Z_c),X_a)\nonumber \\
&=&(T(\varphi \circ \bar{f}_c \circ \varphi)(X_a,Y_b)+T(\varphi \circ f_{\varphi(a+b)})(Z_c)))\nonumber \\
& \oplus &(T(\varphi \circ \bar{f}_a \circ \varphi)(Y_b,Z_c)+ T(\varphi \circ f_{\varphi (b,c)})(X_a))\nonumber \\
&=&(T(\tau_{\varphi(\varphi(b,c),a)}\circ \varphi \circ \bar{f}_c \circ \varphi \circ \bar{f}_b+\tau_{\varphi(\varphi(a,b),c)} \circ \varphi \circ f_{\varphi(b,c)})(X_a)\nonumber \\
&+& T(\tau_{\varphi(\varphi(b,c),a)}\circ \varphi \circ \bar{f}_c \circ \varphi \circ f_a+\tau_{\varphi(\varphi(a,b),c)}\circ \varphi \circ \bar{f}_a \circ \varphi \circ \bar{f}_c)(Y_b)\nonumber \\
&+&(T(\tau_{\varphi(\varphi(b,c),a)} \circ \varphi \circ f_{\varphi(a,b)}+\tau_{\varphi(\varphi(a,b),c)} \circ \varphi \circ \bar{f}_a \circ \varphi \circ f_b)(Z_c)\nonumber \\
&=&T(\sigma_{-1} \circ \varphi \circ \bar{f}_b \circ \varphi \circ f_c)(X_a)+T(\sigma_{-1} \circ  \varphi \circ \bar{f}_{\varphi(c,a)})(Y_b)\nonumber \\
&+&T(\sigma_{-1} \circ \varphi \circ \bar{f}_b \circ \varphi \circ \bar{f}_a)(Z_c)\nonumber \\
&=&T\sigma_{-1}(T\varphi(T\varphi(Z_c,X_a),Y_b))\nonumber \\
&=& -1\bullet T(\varphi(T\varphi(Z_c,X_a),Y_b)).\label{hkk6}
\end{eqnarray}
Eq. (\ref{hkk6}) indicates that the Jacobi identity is satisfied. Therefore (\ref{hkk3}), (\ref{hkk4}), (\ref{hkk5}), and (\ref{hkk6}) imply that $(T(Lie(G)),\oplus,\bullet,T\varphi )$ is a Lie algebra. 
\end{proof}

Since both $T(Lie(G))$ and $Lie(TG)$ have a Lie algebra structure of dimension $2m$, there exist the following relationship 
$\Omega:T(Lie(G)) \to Lie(TG)\label{hkk10}$ given by

$$\Omega(\bar{X},\bar{V})=\displaystyle\sum_{i=1}^m \bar{a}_i X_{i}^{c}+\bar{v}_i X_{i}^{v}\label{gamma}$$ where $\bar{X}= \displaystyle\sum_{i=1}^m \bar{a}_i X_i\in Lie(G)$, and $\bar{V}=\displaystyle\sum_{i=1}^m \bar{v}_i e_i\in \mathbb{R}^m.$

\begin{proposition}\label{newprop4} $\ T(Lie(G))$ is algebraically isomorphic to $\ (Lie(TG),[,])$.
\end{proposition}

\begin{proof} For simplicity, we represent $\{\widetilde{X_i},\widetilde{y_i}\}$ as a basis for $\ T(Lie(G))$ where $\widetilde{X_i}=(X_i,0)$ and $\widetilde{y_i}=(0,e_i)$ \cite{Myarticle}. Since $ X_i=\displaystyle\sum \delta_{j}^{i} X_j$, using (\ref{hkk10}) we have
\begin{eqnarray}
\Omega(\widetilde{X_i})&=&\Omega((\displaystyle\sum_{j=1}^m \delta_{j}^{i} X_j),0)=\displaystyle\sum_{j}^m \delta_{j}^{i} X_{j}^{c}= X_{i}^{c}\label{xi}
\end{eqnarray}
and since $\ e_i=\displaystyle\sum \delta_{j}^{i} e_j$, we have
\begin{eqnarray}
\Omega(\widetilde{y_i})&=&\Omega(0,\displaystyle\sum_{j=1}^m \delta_{j}^{i} e_j)=\displaystyle\sum_{j}^m \delta_{j}^{i} X_{j}^{v}= X_{i}^{v}\label{vi}.
\end{eqnarray}
Eqs. (\ref{xi}) and (\ref{vi}) shows that $\Omega$ is a linear isomorphism. Showing that $\Omega$ is a Lie algebra homomorphism completes the proof. For that, we take two arbitrary elements $ (\bar{X},\bar{V})$, $(\bar{Y},\bar{W})\in T(Lie(G))$ where $\bar{X}=\displaystyle\sum_{i=1}^m a_i X_i$, $\bar{Y}=\displaystyle\sum_{i=1}^m b_i X_i$ and $\bar{V}=\displaystyle\sum_{i=1}^m v_i e_i$, $\bar{W}=\displaystyle\sum_{i=1}^m w_i e_i$. Then, we have
\begin{eqnarray}
T\varphi((\bar{X},\bar{V}),(\bar{Y},\bar{W}))[x_k]&=&(\bar{Y},\bar{W})[x_k \circ \varphi \circ f_{\bar{X}}]+(\bar{X},\bar{V})[x_k \circ \varphi \circ \bar{f}_{\bar{Y}}]\nonumber \\
&=& \displaystyle\sum_{i,j=1}^m (\bar{w}_i \frac{\partial(x_k \circ \varphi \circ f_{\bar{X}})}{\partial x_j}|_{\bar{Y}}+\bar{v}_i \frac{\partial(x_k \circ \varphi \circ \bar{f}_{\bar{Y}})}{\partial x_j}|_{\bar{X}})\nonumber \\
&=& \displaystyle\sum_{i,j=1}^m (\bar{w}_j \bar{a}_i C_{ij}^{k}+\bar{v}_j \bar{b}_i C_{ji}^k)\nonumber \\
&=& \displaystyle\sum_{i,j=1}^m (\bar{w}_j \bar{a}_i-\bar{v}_j \bar{b}_i) C_{ij}^k.\label{hkk11}
\end{eqnarray}
(\ref{hkk11}) implies that 
\begin{equation}
T\varphi((\bar{X},\bar{V}),(\bar{Y},\bar{W}))=(\displaystyle\sum \bar{a}_i\bar{b}_j[X_i,X_j] ,\displaystyle\sum (\bar{w}_j \bar{a}_i -\bar{v}_j \bar{b}_i)C_{ij}^k e_k).\label{hkk12}
\end{equation}
Using Eq. (\ref{hkk12}), we obtain 
\begin{eqnarray}
\Omega(T\varphi((\bar{X},\bar{V}),(\bar{Y},\bar{W})))&=&\displaystyle\sum \bar{a}_i\bar{b}_j C_{ij}^k X_{k}^{c} +\displaystyle\sum (\bar{w}_j \bar{a}_i- \bar{v}_j \bar{b}_i)C_{ij}^k X_{k}^{v}\nonumber \\
&=& \displaystyle\sum \bar{a}_i\bar{b}_j[X_i,X_j]^c+\displaystyle\sum (-\bar{w}_j \bar{a}_i +\bar{v}_j \bar{b}_i)[X_j,X_i]^v\nonumber \\
&=& \displaystyle\sum \bar{a}_i\bar{b}_j[X_i,X_j]^c+\displaystyle\sum (-\bar{w}_j \bar{a}_i +\bar{v}_j \bar{b}_i)[X_j,X_i]^v\nonumber \\
&=& [\Omega(\bar{X},\bar{V}),\Omega(\bar{Y},\bar{W})].\label{hkk13}
\end{eqnarray}
Eq. (\ref{hkk13}) shows that $\Omega$ is a Lie algebra homomorphism, therefore this finishes the proof.
\end{proof}

   So far, we have obtained a Lie algebra structure on $T(Lie(G))$ which is isomorphic to $Lie(TG)$. Now, given a Lie algebra representation $\phi$, we will obtain a new Lie algebra representation of $Lie(TG)$. For this, we first prove the following proposition.


\begin{proposition}\label{newnewprop1}Let $(g,\varphi)$ and $(h,\gamma)$ be two Lie algebras and $F:g\to h$ be a Lie algebra homomorphism. Then $TF$ is a Lie algebra homomorphism.
\end{proposition}
\begin{proof} Since $F$ is a Lie algebra homomorphism, it is a linear function. Then $TF$ is a linear function too\cite{Morimoto:1968}. Showing that $TF$ preserves Lie brackets, we complete the proof.

Since $F$ is a Lie algebra homomorphism, $F(\varphi(X,Y))=\gamma(F(X),F(Y))$ for all $X,Y\in g$. This means
$F\circ \varphi =\gamma \circ (F\times F)$. Then we can write $TF\circ T\varphi =T \gamma \circ (TF \times TF)$. This leads to 
\begin{eqnarray}
TF(T\varphi(\bar{X},\bar{Y}))=(TF\circ T\varphi)(\bar{X},\bar{Y})= T\gamma(TF(\bar{X}),TF(\bar{Y}))\label{hkk30}
\end{eqnarray}
where $\bar{X}\in Tg$ and $\bar{Y}\in Tg$. Therefore (\ref{hkk30}) implies that $TF$ preserves Lie brackets. 
\end{proof}

   Let $J_n$ be the one-to-one Lie group homomorphism which is defined in \cite{Morimoto:1968} from $T(GL(n,\mathbb{R}))$ to $GL(2n,\mathbb{R})$ and $\hat{J_n}:T(Aut(V)) \to Aut(TV)$ which is the one-to-one homomorphism, obtained by $J_n$. It is well known that the differential of a Lie group homomorphism at the identity is a Lie algebra homomorphism. Therefore,  $T(\hat{J_n})_{(I,0)}$  is a Lie algebra homomorphism. Since $\Omega$ and $\Omega'$ are Lie algebra isomorphisms of $T(Lie(G))$ and $T(End(V))$ respectively, it can be defined a Lie algebra representation, that is 
$\widetilde{\phi}=(T\hat{J_n})_{(I,0)}\circ \Omega'\circ T\phi \circ \Omega^{-1}.$  \label{prolongation}


We note that $\widetilde{\phi}$ in equation (\ref{prolongation}) has the following explicit form 
\begin{equation}
\widetilde{\phi}( \displaystyle\sum_{i=1}^m (a_i X_i^V+b_i X_i^C)=
\begin{pmatrix}
\phi(\displaystyle\sum_{i=1}^m (b_i X_i)) & 0\\
\phi(\displaystyle\sum_{i=1}^m (a_i X_i)) & \phi(\displaystyle\sum_{i=1}^m (b_i X_i))
\end{pmatrix}. \label{proliealg}
\end{equation}

\begin{proposition} If $\widetilde{\Phi}$ is the prolongation of $\Phi$, where $\Phi$ be a Lie group representation with corresponding Lie algebra representation $\phi$, then it can be easily seen that the prolongation of $\phi$ is exactly $\widetilde{\Phi}_{*(e,0)}$ i.e., 
\begin{equation}
\widetilde{\phi}=\widetilde{\Phi}_{*(e,0)}
\end{equation}
where $e$ represents identity element of $G$.
\end{proposition}

\section{Conclusion}

\hspace{5mm} In this paper, we have studied prolongations of real representations of Lie algebras. In particular, we have obtained representations of Lie algebras of tangent bundles of Lie groups. First, we have shown that a tangent bundle of a given Lie algebra attains a Lie algebra structure. Then, we have proven that this tangent bundle is algebraically isomorphic to a Lie algebra of the tangent bundle of corresponding Lie group. Using these, we have defined prolongations of representations of Lie algebras. We have also studied the relationship between the prolongations studied here and the prolongations presented in \cite{Myarticle}.

\section{Appendix}

\hspace{5mm} Below, we present some equations/identities that are used in section 3.

 \begin{proposition}\label{newprop1} For all $ a,b,c\in G$ , $\lambda \in \mathbb{R}$, $ X_a,Y_b,Z_c \in T(Lie(G))$, we have following formulas:
\begin{eqnarray}
\varphi \circ f_a &=& \sigma_{-1} \circ \varphi \circ \bar{f}_a\label{eq1} \\
\varphi \circ \bar{f}_b &=& \sigma_{-1} \circ \varphi \circ f_b\label{eq2} \\
\varphi \circ \bar{f}_b \circ \sigma_{\lambda} &=& \sigma_{\lambda} \circ \varphi \circ \bar{f_b} \label{eq3} \\
\varphi \circ f_{\lambda a} &=& \sigma_{\lambda}\circ \varphi \circ f_a \label{eq4} \\
\varphi \circ \bar{f}_c \circ \tau_b &=& \tau_{\varphi(b,c)} \circ \varphi \circ \bar{f}_c \label{eq5} \\
T(\varphi \circ f_a)(Z_c)\oplus T(\varphi \circ f_b)(Z_c) &=& T(\varphi \circ f_{a+b})(Z_c) \label{eq6} 
\end{eqnarray}
\begin{eqnarray*}
T(\tau_{\varphi(\varphi(b,c),a)} \circ \varphi \circ \bar{f}_c \circ \varphi \circ \bar{f}_b+ \tau_{\varphi(\varphi(a,b),c)}\circ \varphi \circ f_{\varphi(b,c)})(X_a)
\end{eqnarray*}
\begin{eqnarray}
=T(\sigma_{-1} \circ \varphi \circ \bar{f}_b \circ \varphi \circ f_c)(X_a)\label{eq7}
\end{eqnarray}
\begin{eqnarray*}
T(\tau_{\varphi(\varphi(b,c),a)} \circ \varphi \circ \bar{f}_c \circ \varphi \circ f_a+\tau_{\varphi(\varphi(a,b),c)} \circ \bar{f}_a \circ \varphi \circ f_b)(Y_b)
\end{eqnarray*}
\begin{eqnarray}
=T(\sigma_{-1} \circ \varphi \circ \bar{f}_{\varphi(c,a)})(Y_b)\label{eq8}
\end{eqnarray}
\begin{eqnarray*}
T(\tau_{\varphi(\varphi(b,c),a)} \circ \varphi \circ f_{\varphi(a,b)}+\tau_{\varphi(\varphi(a,b),c)} \circ \varphi \circ \bar{f}_a \circ \varphi \circ f_b)(Z_c)
\end{eqnarray*}
\begin{eqnarray}
=T(\sigma_{-1} \circ \varphi \circ \bar{f}_b \circ \varphi \circ \bar{f}_a)(Z_c)\label{eq9}
\end{eqnarray} 
\end{proposition}
\noindent where $\varphi $ represents the Lie bracket, $T\varphi$ represents the differential of $\varphi$ and $x_i$ represents coordinate functions of G.

\begin{proof} For all $x\in Lie(G)$, we have
\begin{eqnarray*}
(\varphi \circ f_a)(x) &=& \varphi(a,x)=\sigma_{-1}((\varphi \circ \bar{f}_a)(x))=\sigma_{-1}((\varphi \circ \bar{f}_a)(x))=(\sigma_{-1} \circ \varphi \circ \bar{f}_a)(x).                                     
\end{eqnarray*}
This proves (\ref{eq1}).
Proofs of (\ref{eq2})-(\ref{eq5}) can be easily shown by the similar way. Therefore, we focus on the rest of the proofs. 

Proof of (\ref{eq6}): Using the coordinate functions of $Lie(G)$, we have
\begin{eqnarray*}
(x_i(\tau_{\varphi(b,c)} \circ \varphi \circ f_a)+x_i(\tau_{\varphi(a,c)} \circ \varphi \circ f_b))(X)= x_i(\varphi(b,c)+\varphi(a,X)+\varphi(a,c)+\varphi(b,X))
\end{eqnarray*} 
\noindent for all $X\in Lie(G)$. Since $\varphi(b,c) $ and $\varphi(a,c)$ are constants, we have 

\begin{equation}
\frac{\partial(x_i(\tau_{\varphi(b,c)} \circ \varphi \circ f_a)+x_i(\tau_{\varphi(a,c)} \circ \varphi \circ f_b)}{\partial x_j}|_c=
 \frac{\partial(x_i \circ \varphi \circ f_{a+b})}{x_j}|_c\label{in.eq}.\\
\end{equation}

\noindent Using $Z_c=\displaystyle\sum_{j=1}^m z_j \frac{\partial}{\partial x_j}|_c$ and (\ref {in.eq}), we have
\begin{eqnarray}
(T(\varphi \circ f_a)(Z_c) \oplus T(\varphi \circ f_b)(Z_c))[x_i]&=&(T\tau_{\varphi(b,c)}(T(\varphi \circ f_a)(Z_c))+T\tau_{\varphi(a,c)}(T(\varphi \circ f_b)(Z_c)))[x_i]\nonumber \\
&=&\displaystyle\sum_{j=1}^m z_j \frac{\partial(x_i \circ \tau_{\varphi (b,c)} \circ \varphi \circ f_a+x_i \circ \tau_{\varphi(a,c)} \circ \varphi \circ f_b)}{\partial x_j}|_c \nonumber \\
&=&\displaystyle\sum_{j=1}^m z_j \frac{\partial(x_i \circ \varphi \circ f_{a+b})}{x_j}|_c \nonumber \\
&=&(T(\varphi \circ f_{a+b})(Z_c))[x_i].\label{in.eq2}
\end{eqnarray}
This completes the proof.

Proof of (\ref{eq7}): For all $\ X\in Lie(G)$, we have
\begin{eqnarray}
(x_i \circ \tau_{\varphi(\varphi(b,c),a)} \circ \varphi \circ \bar{f}_c \circ \varphi \circ \bar{f}_b+ x_i \circ \tau_{\varphi(\varphi(a,b),c)} \circ \varphi \circ f_{\varphi(b,c)})(X)\nonumber \\
=x_i(\varphi (\varphi (b,c),a)+\varphi (\varphi (X,b),c)+\varphi (\varphi (a,b),c)+\varphi (\varphi (b,c),X)).\label{hkk1} 
\end{eqnarray} 
Then the differential of (\ref{hkk1})
\begin{eqnarray}
\frac{\partial x_i ( \tau_{\varphi(\varphi(b,c),a)} \circ \varphi \circ \bar{f}_c \circ \varphi \circ \bar{f}_b+ \tau_{\varphi(\varphi(a,b),c)} \circ \varphi \circ f_{\varphi(b,c)})}{\partial x_j} 
=\frac{\partial(x_i \circ \sigma_{-1} \circ \varphi \circ \bar{f}_b \circ \varphi \circ f_c)}{\partial x_j}\nonumber.
\end{eqnarray}
Using (\ref{hkk1}) and above equation, we finish the proof as follows
\begin{eqnarray*}
&T(\tau_{\varphi(\varphi(b,c),a)} \circ \varphi \circ \bar{f}_c \circ \varphi \circ \bar{f}_b+ \tau_{\varphi(\varphi(a,b),c)}\circ \varphi \circ f_{\varphi(b,c)})(X_a)[x_i]\\
=&\displaystyle\sum_{j=1}^m X_j \frac{\partial(x_i \circ \tau_{\varphi(\varphi(b,c),a)} \circ \varphi \circ \bar{f}_c \circ \varphi \circ \bar{f}_b+ \tau_{\varphi(\varphi(a,b),c)} \circ \varphi \circ f_{\varphi(b,c)})}{\partial x_j}|_a\\
=&\displaystyle\sum_{j=1}^m X_j \frac{\partial(x_i \circ \sigma_{-1} \circ \varphi \circ  \bar{f}_b \circ \varphi \circ f_c)}{\partial x_j}|_a\\
=&T(\sigma_{-1} \circ \varphi \circ \bar{f}_b \circ \varphi \circ f_c)(X_a)[x_i].
\end{eqnarray*}
Proof of (\ref{eq8}) and (\ref{eq9}) can be similarly performed.
\end{proof}  

\begin{lemma}\label{newlem3} Let $C_{ij}^k$ be structure constants of $G$, then we have:
\begin{eqnarray}
\frac{\partial (x_k \circ \varphi \circ f_{\bar{X}})}{\partial x_j}|_{\bar{Y}}=\bar{a}_iC_{ij}^{k}\label{eq11}\hspace{3mm}and\hspace{3mm}
\frac{\partial (x_k \circ \varphi \circ \bar{f}_{\bar{Y}})}{\partial x_j}|_{\bar{X}}= \bar{b}_i {C}_{ji}^{k}\label{eq12}
\end{eqnarray}
where $\bar{X}=\displaystyle\sum_{i=1}^m \bar{a}_i X_i$ and $\bar{Y}=\displaystyle\sum_{i=1}^m \bar{b}_i X_i.$
\end{lemma}

\begin{proof} Using $(x_k\circ \varphi \circ f_{\bar{X}})(X)=\bar{a}_ix_tC^k_{it}$  for all $X\in Lie(G)$, we show
\begin{equation*}
\frac{\partial(x_k \circ \varphi \circ f_{\bar{X}})}{\partial x_j}|_{\bar{Y}}=\frac{\partial(\bar{a}_ix_tC^k_{it})}{\partial{x_j}}=\bar{a}_i C^k_{ij}.
\end{equation*}
Proof of Eq. (\ref{eq12}) can be done similarly.
\end{proof}

\bigskip

\end{document}